\documentclass[11pt]{amsart}
\usepackage{preamble}

\author{Noam Kimmel}
\title{Vanishing of {P}oincar\'e series for congruence subgroups}
\address{N. Kimmel: Raymond and Beverly Sackler School of Mathematical Sciences, Tel Aviv University, Tel Aviv 69978, Israel.}
\email{\href{mailto:noamkimmel@mail.tau.ac.il}{noamkimmel@mail.tau.ac.il}}
\thanks{This research was supported by the European Research Council (ERC) under the European Union's  Horizon 2020 research and innovation program  (Grant agreement No.    786758).}

\begin{document}

\maketitle

\begin{abstract}
We consider the problem of the vanishing of {P}oincar\'e series for congruence subgroups. 
Denoting by $P_{k,m,N}$ the {P}oincar\'e series of weight $k$ and index $m$ for the group $\Gamma_0(N)$, we show that for certain choices of parameters $k,m,N$, the {P}oincar\'e series does not vanish.
Our methods improve on previous results of Rankin (1980) and Mozzochi (1989). 
\end{abstract}

\setcounter{tocdepth}{1}
{\tableofcontents}

\section{Introduction}

For $k,m,N \in \NN$, $k$ even, we denote by $P_{k,m,N}(z)$ the Poincar\'e series of weight $k$ and index $m$ at $i\infty$ for 
$$
\Gamma_0(N) = \SET{
\begin{pmatrix}
    a       & b  \\
    c       & d 
\end{pmatrix}\in \SL \; , \; N\mid c
}.
$$
That is, we define
$$
P_{k,m,N}(z) = 
\sum_{\gamma \in \Gamma_\infty \backslash \Gamma_0(N)}
j(\gamma, z)^{-k} e(m\gamma z)
$$
where 
$$
\Gamma_\infty = \left\lbrace \begin{pmatrix}
    1       & *  \\
    0       & 1 
\end{pmatrix} \in \SL \right\rbrace, \quad
j\left(
\begin{pmatrix}
    a       & b  \\
    c       & d 
\end{pmatrix}
,z \right) = cz+d,
$$
and $e(z) = e^{2\pi i z}$.

It is currently not known when these Poincar\'e series identically vanish, a problem which dates back to Poincar\'e’s memoir on Fuchsian groups \cite[Page 249]{MR1554584}.
For level $N=1$ and weight $k=12$, Lehmer (1947) conjectured that these Poincar\'e series never vanish, equivalently that the coefficients $\tau(n)$ of the modular discriminant are never zero \cite{MR0021027}.

For level $N=1$ and large weight $k$, a partial answer to the non-vanishing question was given by Rankin \cite{MR0597120} where he showed that for sufficiently large even $k$ one has $P_{k,m,1}\not\equiv 0$ for all 
$$
m \leq   \exp{-B\frac{\log k}{\log \log k}} k^2
$$
for some absolute constant $B > 0$.
This result was later extended by Lehner \cite{MR0597127} to general Fuchsian groups with a weaker result,
and by Mozzochi \cite{MR0982000} to $P_{k,m,N}$ with $N>1$.

In this paper we improve Rankin's result, proving the non-vanishing of $P_{k,m,1}$ for $m\leq \frac{(k-1)^2}{16\pi^2}$.
We then generalize the method to give an improvement of Mozzochi's results for $N>1$.
In extending our methods from $P_{k,m,1}$ to $P_{k,m,N}$, we are led to the problem of providing lower bounds on certain Kloosterman sums.

\subsection*{Notations}
We use the notation $f(x)\ll g(x)$ to indicate that there is some constant $C>0$ such that $\left|f(x)\right|<Cg(x)$ for all valid inputs $x$. 
If the constant $C$ depends on some parameters, this will be indicated using subscripts such as $f(x) \ll_\epsilon g(x)$.

We will use $\omega(N)$ to denote the number of unique prime factors of $N$.

We will use '$*$' to denote a number in $\ZZ$ where the precise value is not important.

We will denote by $\mu_{ST}$ the Sato-Tate distribution on $[0,\pi]$, that is 
$$
\mu_{ST} = \frac{2}{\pi} \sin^2 \theta d\theta.
$$

Throughout, $k$ will always denote a positive even integer.

\section{Main results}
We will be interested in showing that $P_{k,m,N}\not\equiv 0$ for various choices of parameters $k,m,N$.

We begin with the case $N=1$, for which we prove
\begin{theorem}\label{thm-non_vanishing}
Let $k\in\NN$ be a sufficiently large even integer.
Then for $1\leq m \leq \frac{(k-1)^2}{16\pi^2}$ we have
$$
v_\infty\left(P_{k,m,1}\right) = 1,
$$
and in particular $P_{k,m,1}\not\equiv 0$.
\end{theorem}
Here $v_\infty(f)$ is the order of vanishing of $f$ at the cusp $i\infty$.
This gives an improvement of Rankin's result \cite{MR0597120}.

We then give a generalization to any square-free $N$:
\begin{theorem}\label{thm-KN}
Let $N\in \NN$ be square-free.
Then for all $k\gg_N 1$ one has $v_\infty\left(P_{k,m,N}\right) = 1$ for all 
$$
1\leq m\leq \frac{(k-1)^2 N^2}{16\pi^2},
$$
and as a consequence, $P_{k,m,N}\not\equiv 0$.
\end{theorem}
This theorem gives an improvement of \cite[Theorem 3]{MR0982000} where Mozzochi shows that for $k\gg_N 1$ one has $P_{k,m,N}\not\equiv 0$ for $m \ll_\epsilon k^{1 - \epsilon}N^{1 - \epsilon}$, or $m \ll k^{2 - \epsilon}N^{2 - \epsilon}$ in the case where $\gcd(m,N)=1$ (the precise statement in Mozzochi's work is slightly more involved).

In the case where $N$ is prime, we also give a result about the non-vanishing of $P_{k,m,N}$ for sufficiently large $k$ independent of $N$.
\begin{theorem}\label{thm-Kp}
Let $\epsilon > 0$.
Then for all $k\gg_\epsilon 1$, all primes $p$, and all:
$$
m\ll_\epsilon k^2 p^{\frac{3}{2}-\epsilon},
\quad m\neq p
$$
we have $v_\infty\left(P_{k,m,p}\right) \ll_\epsilon p^{\frac{1}{2} + \epsilon}$ (and in particular $P_{k,m,p}\not\equiv0$).
\end{theorem}

This result can be compared with \cite[Theorem 2]{MR0982000} where  Mozzochi shows that for $k\gg 1$ one has $P_{k,m,N}\not\equiv0$ for $m \ll \exp{-B\frac{\log k}{\log \log k}} k^2$, and its improvement in \cite[Theorem 5.2]{MR2975157} to the range $m \ll N \exp{-B\frac{\log k}{\log \log k}} k^2$.
Thus, \Cref{thm-Kp} provides a further improvement when $N$ is prime, both in the $k$ aspect and in the $N$ aspect.

\Cref{thm-Kp} can also be extended to the case where $N$ is square-free in the following sense:
\begin{theorem}\label{thm-Kr}
Let $\epsilon > 0$, and let $r\in \NN$.
Then for all $k\gg_{\epsilon, r} 1$, all $N=p_1 \cdot p_2\cdot ... \cdot p_r$ with $p_1<p_2<...<p_r$ primes, and all:
$$
m\ll_{\epsilon,r} \frac{ k^2 N^2}{p_r^{\frac{1}{2} + \epsilon}},
\quad
\gcd(m,N) = 1
$$
we have $v_\infty\left(P_{k,m,N}\right) \ll_{\epsilon,r} p_r^{\frac{1}{2} + \epsilon}$ (and in particular $P_{k,m,N}\not\equiv0$).
\end{theorem}

Theorems \ref{thm-Kp} and \ref{thm-Kr} use Katz’s results \cite{MR0955052} regarding the distribution of Kloosterman angles in order to show the existence of large Kloosterman sums.
Using more elementary bounds (looking only at the second moment), we give another version of \Cref{thm-Kr}.
This gives a weaker result, but makes the dependence on $r = \omega(N)$ more concrete.
\begin{theorem}\label{thm-K_elemntary}
Let $N\in\NN$ be square-free.
Then for all $k\gg \omega(N)$ and all $1\leq m\leq \frac{1}{32\pi^2}N (k-1)^2$ with $\gcd(m,N) = 1$ we have $v_\infty\left(P_{k,m,N}\right)\leq 2N$ (and in particular $P_{k,m,N}\not\equiv0$).
\end{theorem}

\section{Preliminary lemmas}

\subsection{Fourier expansion of $P_{k,m,N}$}
We will denote by $\pkm{n}$ the $n$-th Fourier coefficient of $P_{k,m,N}$ so that 
$$
P_{k,m,N}(z) = \sum_{n\geq 1}\pkm{n} e(nz).
$$

It is known that
\begin{equation}\label{eq-pkmn}
\pkm{n}
=
\delta_{m,n}
+ 
2\pi i^{k}
\left(\frac{n}{m}\right)^{\frac{k-1}{2}}
\sum_{c\geq 1}\frac{K(m,n,cN)}{cN}J_{k-1}\left(\frac{4\pi \sqrt{mn}}{cN}\right)  
\end{equation}
where $\delta_{m,n}$ is the Kronecker delta function, $J$ is the Bessel function of the first kind, and 
$K(a,b,c)$ is the Kloosterman sum:
\begin{equation}\label{eq-Kabc}
K(a,b,c) = \sum_{\substack{1 \leq x \leq c \\ \gcd(x,c) = 1}}
e\left(\frac{ax + b \overline{x}}{c}\right)
\end{equation}
where $\overline{x}$ is the inverse of $x$ in $\left(\bigslant{\ZZ}{c\ZZ}\right)^*$.

\subsection{Kloosterman sums}

We will require some known results regarding Kloosterman sums.
First, we note that Kloosterman sums satisfy the following twisted multiplicativity property:
For $c_1,c_2$ with $\gcd(c_1,c_2)=1$ one has
\begin{equation}\label{eq-twisted}
K(a,b,c_1c_2)
=
K(a\overline{c_2},b\overline{c_2},c_1)
K(a\overline{c_1},b\overline{c_1},c_2)
\end{equation}
where $\overline{c_2}$ is the inverse of $c_2$ mod $c_1$ and $\overline{c_1}$ is the inverse of $c_1$ mod $c_2$.

One also has the equality 
\begin{equation}\label{eq-ab}
K(a,bc,N) = K(ac,b,N)
\end{equation}
if $\gcd(c,N) = 1$.

From \eqref{eq-twisted} and \eqref{eq-ab} we also get
\begin{equation}\label{eq-twisted2}
K(a,b,c_1c_2)
=
K(a\overline{c_2}^2,b,c_1)
K(a\overline{c_1}^2,b,c_2)
\end{equation}
whenever $\gcd(c_1,c_2) = 1$.

\begin{lemma}\label{lem-ks_nonzero}
For a square-free $N\in\NN$ and any $a,b\in\ZZ$ one has 
$$
K(a,b,N)\neq 0.
$$
\end{lemma}

\begin{proof}
This is a known result, see for example \cite[Page 63]{MR1474964}.
We present the proof here for convenience.

For a prime $p$, consider $K(a,b,p)$ for some $a,b\in\ZZ$.
We look at $K(a,b,p)$ mod $(1 - \zeta_p)$ in $\QQ(\zeta_p)$, (where $\zeta_p$ is a primitive $p$-th root of unity).
The element $1 - \zeta_p$ is a prime of norm $p$, and we have 
$$
K(a,b,p)\equiv -1 \bmod{(1 - \zeta_p)}.
$$
Thus $K(a,b,p)\neq 0$ for all $a,b\in\ZZ$.
The statement then follows from the twisted multiplicativity of Kloosterman sums \eqref{eq-twisted}.
\end{proof}

\begin{lemma}\label{lem-ks_Nc}
Let $N,c\in \NN$.
Then for $m,n\in \ZZ$ with $\gcd(n,N) = g$ we have
$$
\left|K(m,n,Nc)\right| \leq 2^{\omega(N) + \frac{1}{2}} c \sqrt{N} \sqrt{g}.
$$
\end{lemma}

\begin{proof}
Write $c = ds$ where $\gcd(s,N) = 1$ and $p\mid d \Rightarrow p\mid N$.
Then from the twisted multiplicativity of Kloosterman sums \eqref{eq-twisted2} we get
$$
K(m,n,cN)
=
K(*,*,s) K(*,n, dN).
$$
For $K(*,*,s)$ we use the trivial bound $|K(*,*,s)| \leq s$.
For $K(*,n, dN)$, using $\omega(dN) = \omega(N)$ and the fact that $\gcd(n,dN) \leq \gcd(n,N)\gcd(n,d)\leq gd$, we have 
$$
|K(*,n, dN)| \leq 2^{\omega(N) + \frac{1}{2}} \sqrt{dN} \sqrt{gd}
= 2^{\omega(N) + \frac{1}{2}} d \sqrt{N} \sqrt{g}
$$
(see \cite[Lemma 3.1]{MR0597120}).
Combining these bounds, we get
$$
|K(m,n,cN)| \leq 2^{\omega(N) + \frac{1}{2}} sd \sqrt{N}\sqrt{g}
= 2^{\omega(N) + \frac{1}{2}} c \sqrt{N} \sqrt{g}. 
$$
\end{proof}

\begin{lemma}\label{lem-K_phiN}
Let $m,N\in\NN$, $N$ square-free, $\gcd(m,N) = 1$.
Then there exists some $1\leq n\leq 2N$ such that $n\neq m$, $\gcd(n,N)=1$ and
$$
\left|K(m,n,N)\right| \geq \frac{\sqrt{N}}{2^{\frac{1}{2}\omega(N)+1}}.
$$
\end{lemma}
\begin{proof}
We begin with the well known computation of the second moment of Kloosterman sums mod $N$.
For $m\in \ZZ$ we denote
$$
S_2(m;N) = \sum_{n\in\left(\bigslant{\ZZ}{N\ZZ}\right)^* }K\left(m,n,N\right)^2.
$$
From the twisted multiplicativity of Kloosterman sums \eqref{eq-twisted2}, we have that $S_2(m;N)$ is multiplicative in $N$.
Furthermore, for a prime $p$ with $p\nmid m$ we have 
$$
S_2(m;p) = p^2 - p - 1 
$$
(see \cite[Section 4.4]{MR1474964}).
Since $p^2 - p - 1 \geq \frac{1}{2}p^2 $ for $p\geq 3$, it follows that for square-free $N$ we have $S_2(m;N) \geq \frac{N^2}{2^{\omega(N) + 1}}$.
This implies that there is some $1 \leq n \leq N$ with $\gcd(n,N) = 1$ such that 
$$
|K(m,n,N)| \geq \frac{\sqrt{N}}{2^{\frac{1}{2}\omega(N)+1}}.
$$
We can further ensure that $n\neq m$ by replacing $n$ with $n+N$ if necessary, so that we still have $n\leq 2N$.
\end{proof}

We will also require the result of Katz regarding the distribution of Kloosterman angles, and its generalization via the P\'olya–Vinogradov method to short intervals.

\begin{lemma}\label{lem-ks_ST}
For a prime $p$ and some $a,b\in \ZZ$, $p\nmid ab$, there exists $\theta_{p,ab} \in [0,\pi]$ such that
$$
K(a,b,p) = 2\sqrt{p} \cos\left(\theta_{p,ab}\right).
$$
Let $\epsilon > 0$, then for any $I(p) \geq p^{\frac{1}{2} + \epsilon}$ and any $m(p)$ with $p\nmid m$, the angles
$$
\SET{\theta_{p,mn} \; : \; 1\leq n \leq I(p)}
$$
become equidistributed according to the Sato-Tate measure $\mu_{ST}$ as $p\rightarrow\infty$.
\end{lemma}
\begin{proof}
As stated above, this follows from Katz's result on the distribution of Kloosterman angles \cite{MR0955052}, and a variant of the P\'olya–Vinogradov method.
See for example \cite{MR1345284}, Specifically Proposition 2 and Corollary 2.10 with the preceding discussion.
\end{proof}

\subsection{Bessel functions}
We will also require some bounds for the $J$ Bessel function.

\begin{lemma}\label{lem-J_lower}
For $\nu\geq 0$, $0< \delta \leq 1$ we have
$$
J_\nu\left(\nu \delta\right)
\gg \nu^{-\frac{1}{3}} \delta^{\nu}.
$$
\end{lemma}
\begin{proof}
we use the following bounds:
$$
J_\nu(\nu \delta) \geq J_\nu(\nu) \delta^\nu
$$
valid for all $0<\delta\leq 1$, $\nu \geq 0$ \cite[\href{https://dlmf.nist.gov/10.14\#E7}{(10.14.7)}]{NIST:DLMF}.
We also have that
$$
J_\nu(\nu) =  \frac{\Gamma\left(\frac{1}{3}\right)}{48^{\frac{1}{6}}\pi} \nu^{-\frac{1}{3}} + \BigO{\nu^{-\frac{5}{3}}}
$$
\cite[eq.(2), pg. 232]{MR0010746}.
Combining these bounds gives the required result.
\end{proof}

\begin{lemma}\label{lem-J_upper}
For $\nu \geq 1$ and $\delta \geq 0$ we have
$$
\left|J_\nu\left(\nu\delta\right)\right|
\ll \nu^{-\frac{1}{2}}\left(\frac{e}{2}\delta\right)^\nu. 
$$
\end{lemma}
\begin{proof}
This is \cite[Lemma 4.1]{MR0597120}.
\end{proof}

\begin{lemma}\label{lem-Jsum_upper}
For $\nu \geq 2$, $\delta \geq 0$ and $c_0\in\NN$ we have
\begin{equation*}
\sum_{c \geq c_0}\left|J_\nu\left(\nu \frac{\delta }{c}\right)\right|
\ll \nu^{-\frac{1}{2}}\left(\frac{e}{2 c_0}\delta\right)^\nu.
\end{equation*}
\end{lemma}

\begin{proof}
Using the upper bound from \cref{lem-J_upper}, we get
\begin{align*}
\sum_{c \geq c_0}\left|J_\nu\left(\nu \frac{\delta }{c}\right)\right|
&\ll
\nu^{-\frac{1}{2}}\left(\frac{e}{2}\delta\right)^\nu 
\sum_{c\geq c_0} c^{-\nu}
\\&\ll
\nu^{-\frac{1}{2}}\left(\frac{e}{2}\delta\right)^\nu 
\left(c_0^{-\nu} + \int_{c_0}^{\infty}x^{-\nu}dx\right)
\\&\ll
\nu^{-\frac{1}{2}}\left(\frac{e}{2 c_0}\delta\right)^\nu.
\end{align*}
\end{proof}

\section{Proofs of main results}

We begin by proving \Cref{thm-KN}.
\Cref{thm-non_vanishing} then follows as a special case by taking $N=1$.
\begin{proof}[Proof of \Cref{thm-KN}]
Let $k\in2\NN$, $N\in\NN$ square-free, and 
$$
1 \leq m \leq \frac{(k-1)^2 N^2}{16\pi^2}.
$$
We wish to show that $v_\infty(P_{k,m,N}) = 1$ for sufficiently large $k$ (depending on $N$), or equivalently that $\pkm{1} \neq 0$.

Assume first that $m > 1$.
In this case, from \eqref{eq-pkmn} we have
$$
\pkm{1}
= 2\pi i^k m^{-\frac{k-1}{2}}
\sum_{c\geq 1}\frac{K(m,1,cN)}{cN}J_{k-1}\left(\frac{4\pi \sqrt{m}}{cN}\right) . 
$$
And so, showing $\pkm{1}\neq 0$ is equivalent to showing that the sum
$$
S = \sum_{c\geq 1}\frac{K(m,1,cN)}{cN}J_{k-1}\left(\frac{4\pi \sqrt{m}}{cN}\right)
$$
is non-zero.
We denote $\nu = k-1$, $\delta = \frac{4\pi \sqrt{m}}{ N \nu}$ so that we have
$$
S = \sum_{c\geq 1}\frac{K(m,1,cN)}{cN}J_{\nu}\left(\nu\frac{\delta}{c}\right).
$$
Note also that from our choice of $m$ we have $\delta \leq 1$.

We begin by considering the first summand in $S$ corresponding to $c=1$.
From \Cref{lem-ks_nonzero} we have that $K(m,1,N)\neq 0$ for all $m\in\ZZ$.
Denote 
$$
\epsilon_N = \mathop{\mathrm{min}}_{m}\left| K(m,1,N)\right|.
$$
Then we have, using \Cref{lem-J_lower}, that 
$$
\frac{K(m,1,N)}{N}J_{\nu}\left(\nu\delta\right)
\gg
\frac{\epsilon_N}{N} \nu^{-\frac{1}{3}}
\delta^\nu.
$$

As for the rest of the summands in $S$, using the trivial bound 
$$
|K(m,1,cN)|\leq cN
$$
and \Cref{lem-Jsum_upper} we get
$$
\sum_{c\geq 2}\frac{K(m,1,cN)}{cN}J_{\nu}\left(\nu\frac{\delta}{c}\right)
\ll
\sum_{c\geq 2}
\left|J_{\nu}\left(\nu\frac{\delta}{c}\right)\right|
\ll
\nu^{-\frac{1}{2}}\left(\frac{e}{4}\delta\right)^\nu.
$$

Since $\frac{e}{4} < 1$, we have that for sufficiently large $k$ (in terms of $N$) the lower bound that we got for the first summand will be larger than the upper bound we got for the rest of the sum.
And so, for $k\gg_N 1$ we have $\pkm{1}\neq 0$.

We now consider the case $m=1$.
In this case we have
$$
p_{k,N}(1;1)
= 1 + 2\pi i^k \sum_{c\geq 1}\frac{K(1,1,Nc)}{Nc}J_{k-1}\left(\frac{4\pi}{Nc}\right).
$$
Denote $\nu = k-1$ and $\delta = \frac{4\pi}{N\nu}$.
Using the trivial bound $|K(1,1,Nc)|\leq Nc$ and \Cref{lem-Jsum_upper} we get
$$
\left|\sum_{c\geq 1}\frac{K(1,1,Nc)}{Nc}J_{k-1}\left(\frac{4\pi}{Nc}\right)\right|
\leq
\sum_{c\geq 1} \left| J_\nu\left(\nu \frac{\delta}{c}\right)\right|
\leq
\nu^{-\frac{1}{2}}\left(\frac{2 e\pi}{N\nu}\right)^\nu.
$$
This tends to 0 as $k\rightarrow\infty$.
It follows that $p_{k,N}(1;1) \neq 0$ for large enough $k$.
\end{proof}

We now give a proof of \Cref{thm-K_elemntary}.
\begin{proof}
Let $m,N\in\NN$ with $N$ square-free, $1\leq m\leq \frac{1}{32\pi^2}N (k-1)^2$ and $\gcd(m,N)= 1$.
From \Cref{lem-K_phiN} there exists some $1\leq n\leq 2N$ satisfying $\gcd(n,N)=1$, $n\neq m$ and 
$$
\left|K(m,n,N)\right| \geq \frac{\sqrt{N}}{2^{\frac{1}{2}\omega(N)+1}}.
$$

We consider $\pkm{n}$:
$$
\pkm{n}
= 2\pi i^k \left(\frac{n}{m}\right)^{\frac{k-1}{2}}
\sum_{c\geq 1}\frac{K(m,n,cN)}{cN}J_{k-1}\left(\frac{4\pi \sqrt{mn}}{cN}\right).
$$
Thus, it is enough to show that the sum above is non-zero.
The first term in the sum (corresponding to $c=1$) can be bounded from below using \Cref{lem-J_lower}:
$$
\frac{K(m,n,N)}{N}J_{\nu}\left(\nu \delta\right)
\gg 
\frac{1}{N}\frac{\sqrt{N}}{2^{\frac{1}{2}\omega(N)}}\nu^{-\frac{1}{3}}\delta^{\nu}
=
\frac{1}{\sqrt{N}2^{\frac{1}{2}\omega(N)}}\nu^{-\frac{1}{3}}\delta^{\nu}
$$
where $\nu = k-1$, $\delta = \frac{4\pi \sqrt{mn}}{N(k-1)}$, and we have $\delta \leq 1$ from $n\leq 2N$ and our restriction on $m$.

As for the rest of the sum, using \Cref{lem-ks_Nc} and \Cref{lem-Jsum_upper} we get
$$
\sum_{c\geq 2}\frac{K(m,n,cN)}{cN}J_{\nu}\left(\nu \frac{\delta}{c}\right)
\ll
\frac{\sqrt{N}2^{\omega(N)}}{N} 
\sum_{c\geq 2}\left|J_{\nu}\left(\nu \frac{\delta}{c}\right)\right|
\ll
\frac{2^{\omega(N)}}{\sqrt{N}} \nu^{-\frac{1}{2}} \left(\frac{e}{4}\delta\right)^{\nu}.
$$
It follows that if $\nu$ is large enough so that $\left(\frac{e}{4}\right)^{\nu} \ll 2^{-\frac{3}{2}\omega(N)}$ then the lower bound we got for the first summand will be larger than the upper bound we got for the rest of the sum, and the result follows.

\end{proof}

We now prove \Cref{thm-Kr}, \Cref{thm-Kp} then follows as a special case by taking $r=1$.

\begin{proof}
Let $N=p_1 \cdot p_2\cdot ... \cdot p_r$ with $p_1<p_2<...<p_r$ primes.  
We begin by considering $K(m,n,N)$ for various $n$'s satisfying $n\ll_{\epsilon,r}p_r^{\frac{1}{2} + \epsilon}$.

From the twisted multiplicativity of Kloosterman sums \eqref{eq-twisted2}, there exist integers $m_1,m_2,...,m_r$ such that
\begin{equation}\label{eq-kmnN_tm}
K(m,n,N)
=
\prod_{i=1}^{r}K(m_i, n, p_i).    
\end{equation}
We note that the condition $\gcd(m,N)=1$ further implies $\gcd(m_i,p_i)=1$.

Let $\delta_r$ be some small positive constant such that
$$
\mathrm{Prob}\left(|\cos(\theta)| > \delta_r\; , \; \theta\sim \mu_{ST}\right)
>
1 - \frac{1}{r+1}.
$$

From \Cref{lem-ks_ST} it follows that for all sufficiently large primes $p_i$, the proportion of $n$'s satisfying $n\ll_{\epsilon,r}p_r^{\frac{1}{2} + \epsilon}$
such that $K(m_i,n,p_i) \gg 2\delta_r \sqrt{p_i} $ is at least $1 - \frac{1}{r+1}$.

There might be a finite set of primes $\mathcal{Q}_{r,\epsilon}$ for which this is not true.
However, if $q\in\mathcal{Q}_{r,\epsilon}$ is such a prime, we know from \Cref{lem-ks_nonzero} that $K(*,*,q)$ is never zero.
Thus, we can replace $\delta_r$ with a smaller positive constant $\delta_{r,\epsilon}$ such that $|K(a,b,q)| \geq 2\delta_{r,\epsilon} \sqrt{q} $ for all $q\in \mathcal{Q}_{r,\epsilon}$ and all $a,b$.

And so, we conclude that there is some $\delta_{r,\epsilon}>0$ such that $|K(m_i,n,p_i)|\geq 2\delta_{r,\epsilon} \sqrt{p_i}$ for a proportion of at least $1 - \frac{1}{r+1}$ of all $n\ll_{\epsilon, r} p_r^{\frac{1}{2} + \epsilon}$.

From the pigeonhole principle, it follows that there is some $n\ll_{\epsilon, r} p_r^{\frac{1}{2} + \epsilon}$ such that 
$$
|K(m_i,n,p_i)|\geq 2\delta_{r,\epsilon} \sqrt{p_i}
$$
for all $1\leq i \leq r$.
It follows from \eqref{eq-kmnN_tm} that 
$$
|K(m,n,N)|\geq \left(2\delta_{r,\epsilon}\right)^r \sqrt{N}.
$$
In fact, since there is a positive proportion of such $n$'s, we can further add the restrictions that $n\neq m$.
We also note that from our construction of $n$ we have that $\gcd(n,N)$ is divisible only by primes from $\mathcal{Q}_{r,\epsilon}$.
Since $\mathcal{Q}_{r,\epsilon}$ depends only on $r,\epsilon$, we have that $\gcd(n,N) \ll_{\epsilon, r} 1$ (since $N$ is square-free).

For the $n$ we chose, we consider $\pkm{n}$.
We have that 
$$
\pkm{n}
= 2\pi i^k \left(\frac{n}{m}\right)^{\frac{k-1}{2}}
\sum_{c\geq 1}\frac{K(m,n,cN)}{cN}J_{k-1}\left(\frac{4\pi \sqrt{mn}}{cN}\right) . 
$$
Denote $\nu = k-1$ and $\delta = \frac{4\pi \sqrt{mn}}{\nu N}$.
The restriction $m\ll_{\epsilon,r} \frac{k^2 N^2}{p_r^{\frac{1}{2} + \epsilon}}$ ensures that $\delta < 1$ since we have $n\ll_{\epsilon,r} p_r^{\frac{1}{2} + \epsilon}$.
In order to show that $\pkm{n}\neq 0$ it is enough to show that the sum
$$
S = \sum_{c\geq 1}\frac{K(m,n,cN)}{cN}J_{\nu}\left(\nu\frac{\delta}{c}\right) 
$$
is non-zero.

We begin by giving a lower bound on the first summand in $S$ corresponding to $c=1$.
We have shown that $K(m,n,N)\gg_{\epsilon,r} \sqrt{N}$.
From this and from \Cref{lem-J_lower} we get
\begin{equation}\label{eq-lower_kmnNJ}
\frac{K(m,n,N)}{N}J_{\nu}\left(\nu\delta\right) 
\gg_{\epsilon,r}
\frac{1}{\sqrt{N}}\nu^{-\frac{1}{3}}\delta^{\nu}.
\end{equation}

We now consider the rest of the summands in $S$.
Using \Cref{lem-ks_Nc}, and the fact that $\gcd(m,N) \ll_{\epsilon,r} 1$, we get:
$$
\sum_{c\geq 2}\frac{K(m,n,cN)}{cN}J_{\nu}\left(\nu\frac{\delta}{c}\right)
\ll_{\epsilon,r}
\sum_{c\geq 2} \frac{1}{\sqrt{N}} \left|J_{\nu}\left(\nu\frac{\delta}{c}\right)\right|
$$
Using \Cref{lem-Jsum_upper} we then have
$$
\sum_{c\geq 2}\frac{K(m,n,cN)}{cN}J_{\nu}\left(\nu\frac{\delta}{c}\right)
\ll_{\epsilon,r}
\frac{1}{\sqrt{N}} \nu^{-\frac{1}{2}}
\left(\frac{e}{4}\delta\right)^\nu.
$$

For sufficiently large $k$ (in terms of $\epsilon,r$) this upper bound will be smaller than the lower bound we got for the first summand \eqref{eq-lower_kmnNJ}.
Thus, for $k\gg_{\epsilon,r} 1$ we conclude that $\pkm{n} \neq 0$.

\end{proof}

\bibliographystyle{plain}
\bibliography{my_bib}

\end{document}